\documentclass[a4paper]{article}
\usepackage{natbib}
\usepackage{stmaryrd,amssymb,amsfonts,amsmath}
\usepackage{amsthm}
\usepackage{hyperref}

\newcommand{\tgerman}[1]{\emph{#1}}

\newcommand{\bbC}{\mathbb{C}}
\newcommand{\bbQ}{\mathbb{Q}}
\newcommand{\bbR}{\mathbb{R}}

\theoremstyle{theorem}
\newtheorem{theorem}{Theorem}
\newtheorem{lemma}[theorem]{Lemma}

\theoremstyle{definition}

\author{David Monniaux}
\title{Fatal Degeneracy in the Semidefinite Programming Approach to the Decision of Polynomial Inequalities}

\begin{document}
\maketitle

\begin{abstract}
In order to verify programs or hybrid systems, one often needs to prove that certain formulas are unsatisfiable. In this paper, we consider conjunctions of polynomial inequalities over the reals. Classical algorithms for deciding these not only have high complexity, but also provide no simple proof of unsatisfiability. Recently, a reduction of this problem to semidefinite programming and numerical resolution has been proposed. In this article, we show how this reduction generally produces degenerate problems on which numerical methods stumble.
\end{abstract}


\section{Introduction}
In order to verify properties of computer programs, hybrid systems \citep{DBLP:conf/hybrid/2005}, or even biological systems \citep{DBLP:conf/hybrid/GhoshTT03}, one often needs to prove that a formula is unsatisfiable. For instance, we may wish to prove that a set of states $I$ is a program invariant: this means that there is no pair of states $(s,s')$ such that $s \rightarrow_\tau s'$, $s \in I$ and $s' \notin I$, where $\tau$ is the transition relation. A proof of unsatisfiability leads to a proof of program safety. A satisfiability witness (a pair of states $s$, $s'$) may be provided to the user as a counterexample to the $I$ ``invariant candidate''; it may also be used for automated counterexample-based refinement \citep{DBLP:conf/popl/BallR02,DBLP:conf/cav/2000}.

In general, formulas arising from program verification contain both conjunctions and disjunctions. One can reduce satisfiability of such formulas to the satisfiability of conjunctions by applying distributivity, but this usually leads to blowup. A better approach is \emph{satisfiability modulo theory}: the problem is reduced to boolean satisfiability testing (SAT), a NP-complete problem for which there exist practically efficient procedures, with the addition of \emph{theory lemmas} stating that certain conjunctions are not satisfiable \citep{Ganzingeretal2004CAV}. This approach relies on the availability of a (fast) decision procedure for conjunctions, which, ideally, given a contradictory conjunction, outputs a contradictory subset.

In this article, we consider the following problem: given a conjunction of polynomial equalities, and (wide and strict) polynomial inequalities, with integer or rational coefficients, decide whether this conjunction is satisfiable\,; that is, whether one can assign real values to the variables so that the conjunction holds.
 
The decision problem for real polynomial inequalities can be reduced to \emph{quantifier elimination}: given a formula $F$, whose atomic formulas are polynomial (in)equalities, containing quantifiers, provide another, equivalent, formula $F'$, whose atomic formulas are still polynomial (in)equalities, containing no quantifier. Quantifier elimination over a formula where all variables are existentially quantified yields an equivalent ground formula (a formula with no variable), whose truth can be decided by mere syntactic evaluation. An algorithm for quantifier elimination over the theory of \emph{real closed fields} (roughly speaking, $(\bbR,0,1+,×,\leq)$ was first proposed by \citet{Tarski51}, but this algorithm had non-elementary complexity and thus was impractical. Later, the \emph{cylindrical algebraic decomposition} (CAD) algorithm was proposed by \citet{Collins75}, with a doubly exponential complexity, but despite improvements \citep{Collins98} CAD is still slow in practice and there are few implementations available.

While quantifier elimination provides a procedure for deciding the satisfiability of quantifier-free constraint systems, it is not the only method for doing so.
\citet[Theorem~3]{BasuPollackRoy96} proposed a satisfiability testing algorithm with complexity $s^{k+1} d^{O(k)}$, where $s$ is the number of distinct polynomials appearing in the formula, $d$ is their maximal degree, and $k$ is the number of variables. We know of no implementation of that algorithm.
\citet{Tiwari05:CSL} proposed an algorithm based on rewriting systems that is supposed to answer in reasonable time when a conjunction of polynomial inequalities has no solution.


Many of the algebraic algorithms are complex, which leads to complex implementatyions. This poses a methodology problem for program verification: can one trust them? Can one rely on a complex verification system in order to prove that a complex program contains no bugs? We could either prove correct the implementation of the algorithm using a proof assistant, 
or we could arrange for the decision procedure to provide a \emph{witness} of its result. The answer of the procedure is correct if the witness is correct, and correctness of the witness can be checked by a simple procedure. We know how to provide unsatisfiability witnesses for systems of complex equalities or linear rational inequalities. It is therefore tempting to seek unsatisfiability witnesses for systems of polynomial inequalities.

\citet{harrison-sos,Parrilo_PhD} have suggested looking for proof witnesses whose existence is guaranteed by the \tgerman{Positivstellensatz}~\citep{Stengle73}. These witnesses involve sums of squares of polynomials, which are obtained as solutions of a \emph{semidefinite programming} (SDP) problem, solved by numerical methods.

In this article, we show how the reduction of the problem of finding \tgerman{Positivstellensatz} witnesses to semidefinite programming leads, in general, to degenerate cases that cannot be solved numerically. It is possible to recast the problem in lower dimension so as to remove degeneracy, but doing so involves computing the Zariski closure of the solution set, by algebraic methods. Since this is as complex as finding a solution point to the SDP problem by algebraic method, we gain nothing by using numerical solver.

We therefore conclude that, though promising it may have seemed, finding \tgerman{Positivstellensatz} witnesses through semidefinite programming numerical techniques is impractical.

\section{Unsatisfiability Witnesses}
For some interesting theories, it is trivial to check that a given valuation of the variables satisfies the formula (e.g., for linear inequalities over the rationals, it suffices to replace the variables by their value and evaluate the form). A satisfiability decision procedure will in this case tend to seek a \emph{satisfiability witness} and provide it to the user when giving a positive answer

In contrast, if the answer is that the problem is not satisfiable, the user has to trust the output of the satisfiability testing algorithm, the informal meaning of which is ``I looked carefully everywhere and did not find a solution.''. In some cases, it is possible to provide \emph{unsatisfiability witnesses}: solutions to some form of dual problem that show that the original problem had no solution. In order to introduce the \tgerman{Positivstellensatz} approach, we first briefly explain two simpler, but similar, problems with unsatisfiability witnesses.



\subsection{Linear Inequalities}\label{part:linear_ineq_witnesses}
Let $C$ be a conjunction of (strict or wide) linear inequalities. A satisfiability witness is just a valuation such that the inequalities hold, and can be obtained by linear programming for instance.

Can we also have unsatisfiability witnesses? For the sake of simplicity, let us consider the case where all the inequalities are wide and take $C$ to be $L_1(x_1, \dots, x_m) \geq 0 \wedge \dots \wedge L_n(x_1, \dots, x_m) \geq 0$ where the $L_i$ are affine linear forms. Obviously, if $\alpha_1, \dots, \alpha_n$ are nonnegative coefficients, then if $C$ holds, then $\sum \alpha_i L_i (x_1, \dots, x_m) \geq 0$ also holds. Thus, if one can exhibit $\alpha_1, \dots, \alpha_n \geq 0$ such that $\sum \alpha_i L_i = -1$  --- otherwise said, a nonnegative linear combination of the inequalities is a \emph{trivial contradiction} ---, then $C$ does not hold. The vector $(\alpha_1, \dots, \alpha_n)$ is thus an unsatisfiability witness.

This refutation method is evidently \emph{sound}, that is, if such a vector can be exhibited, then the original problem had no solution. It is also \emph{complete}: one can always obtain such a vector if the original problem $C$ is unsatisfiable, from Farkas' lemma \citep[\S 6.4, theorem~6]{Dantzig}. A constructive proof of the same fact can be obtained by considering the result of the Fourier-Motzkin algorithm \citep[\S 4.4]{Dantzig} applied to all variables: it outputs a conjunction of variable-free formulas, equivalent to $C$ and obtained by nonnegative linear combinations of the $L_i$. $C$ is unsatisfiable if and only if at least one of these variable-free positive linear combinations is absurd, and this one provides a witness.

Interestingly, the witness is obtained as a solution of a \emph{dual problem} of the same nature as the original problem. That is, the unsatisfiability witness is itself the solution of a system of linear equalities and inequalities... which can be solved by linear programming.

\subsection{Complex Polynomial Equalities}
Let $C$ be a conjunction of polynomial equalities $P_1(x_1, \dots, x_m) =0 \wedge \dots \wedge P_n(x_1, \dots, x_m) =0$ whose coefficients lie in a subfield $K$ (say, the rational numbers $\bbQ$) of an algebraically closed field $K'$ (say, the complex numbers $\bbC$). $C$ is said to be satisfiable if one can find a valuation in $K'$ of the variables in $C$ such that the equalities hold. Such a valuation thereby constitutes a satisfiability witness.

Let us first remark that it is insufficient to look for the coefficients of the satisfiability witness inside $K$: for instance, $X^2=2$ has no rational solutions, but has real solutions $X=\pm \sqrt{2}$. Worse, it is a fact of Galois theory that the solutions of polynomials of degree higher than four cannot be in general expressed using arithmetic operators and $n$-th degree roots. Satisfiability witnesses may thus have to be expressed using general algebraic roots, and checking them is somewhat complex algorithmically.

In contrast, one can get unsatisfiability witnesses that are checkable using simple methods, only involving adding and multiplying polynomials over~$K$. Obviously, if one can find $Q_1,\dots,Q_n \in K[x_1,\dots,x_m]$ such that $\sum_i P_i Q_i=1$, then $C$ has no solution. Again, this method of finding a trivial contradiction is both sound and complete for refutation. The completeness proof relies on a theorem known as \tgerman{Nullstellensatz}:

\begin{theorem}[Hilbert]
Let $K'$ be an algebraically closed field, let $I$ be an ideal in $K'[x_1,\dots,x_n]$. Let $P$ be a polynomial in $K'[x_1,\dots,x_n]$. $P$ vanishes over the common zeroes of the ideals in $I$ if and only if some nonnegative power of $P$ lies in $I$.
\end{theorem}

Apply that theorem to $P=1$ and $I$ the ideal generated by $P_1,\dots,P_m$. $P=1$ vanishes over the common zeroes of $I$ if and only if they have no common zeroes, and, by the theorem, if and only if $1$ lies in $I$, that is, there exists $\bar{Q}_1,\dots,\bar{Q}_m \in K'[x_1,\dots,x_n]$ such that $\sum_i \bar{Q}_i P_i = 1$. $K'$ is a vector space over $K$, thus $K$ has a supplemental space $S$ in $K'$. By projecting the coefficients of the $\bar{Q}_i$ onto $K$, one obtains polynomials $Q_i \in K[x_1,\dots,x_n]$ such that $\sum_i Q_i P_i = 1$. Those $Q_i$ constitute a unsatisfiability witness for~$C$.

For the sake of brevity, the remainder of the explanations will be somewhat sketchy; the reader can refer to e.g. \citet{CoxLittleOShea} if needed.
By Buchberger's algorithm, or some other algorithm, one can compute a Gr\"obner basis $P'_1,\dots,P'_m$ from the $P_1,\dots,P_m$. The ideals generated from both sets are identical, but the Gr\"obner basis has the property that a polynomial lies in the generated ideal if and only if the remainder of its division by the Gr\"obner basis, through the multivariate division algorithm, is null if and only if that polynomial belongs to the ideal. We therefore have a method for testing whether an unsatisfiability witness exists. Furthermore, if it exists, the division algorithm will provide $Q'_1, \dots, Q'_{m'}$ such that $\sum_j P'_j Q'_j =0$. If one has kept track of how the $P'_j$ can be expressed in terms of the $P_j$, then one can compute the $Q_1, \dots, Q_m$ witness.

Note that this algorithm is sound \emph{but incomplete} when $K'$ is not algebraically closed (e.g. the real field~$\bbR$). For instance, the polynomial $x^2+1$ has no real solution, yet the polynomial $1$ is not a member of the ideal generated by it. Thus, Gr\"obner basis computations can provide unsatisfiability witnesses for some systems of polynomial equalities over the reals, but not for all. The real case is much more complex than the complex case.

\section{Polynomial Inequalities}
For the sake of simplicity, we shall restrict ourselves to wide inequalities (the extension to mixed wide/strict inequalities is possible).
Let us first remark that the problem of testing whether a set of wide inequalities with coefficients in a subfield $K$ of the real numbers is satisfiable over the real numbers is equivalent to the problem of testing whether a set of \emph{equalities} with coefficients $K$ is satisfiable over the real numbers. The proof is simple: for each inequality $P_i(x_1,\dots,x_m) \geq 0$, replace it by $P_i(x_1,\dots,x_m) - \mu_i^2 =0$, where the $\mu_i$ are new variables.
One therefore does not gain theoretical simplicity by restricting oneself to inequalities.

\subsection{Real Nullstellensatz and Positivstellensatz}
\citet{Stengle73} proved two theorems regarding the solution sets of systems of polynomial equalities and inequalities over the reals (or, more generally, over real closed fields): a \tgerman{Nullstellensatz} and a \tgerman{Positivstellensatz}. Without going into overly complex notations, let us state consequences of these theorems. Let $K$ be an ordered field (such as~$\bbQ$) and $K'$ be a real closed field containing $K$ (such as the real field~$\bbR$). The corollary of interest to us is:

\begin{theorem}\label{th:real-nullstellensatz}
Let $Z_1,\dots,Z_{n_z}$, $P_1, \dots, P_{n_p}$ be two (possibly empty) sets of polynomials in $K[x_1, \dots, x_m]$. Then
$Z_1(x_1, \dots, x_m)=0 \wedge \dots \wedge Z_{n_z}(x_1, \dots, x_m)=0 \wedge
 P_1(x_1, \dots, x_m) \geq 0 \wedge \dots \wedge P_{n_p}(x_1, \dots, x_m) \geq 0$ has no solution if and only if there exist some polynomials $A$ and $B$ such that $A+B=1$, $A \in I(Z_1,\dots,Z_{n_z})$ and $B \in S(P_1, \dots, P_{n_p})$, where $I(Z_1,\dots,Z_{n_z})$ is the ideal generated by the $Z_1,\dots,Z_{n_z}$ and $S(P_1, \dots, P_{n_p})$ is the semiring generated by the positive elements of $K$ and $P_1^2, \dots, P_{n_p}^2$.
\end{theorem}

Note that this result resembles the one used for linear inequalities (Section~\ref{part:linear_ineq_witnesses}), replacing nonnegative numbers by sums of squares of polynomials.

For a simple example, consider the following system, which obviously has no solution:
\begin{equation}\label{eqn:very_simple_example}
\left\{
\begin{array}{l}
-2 + y^2 \geq 0\\
1 - y^4 \geq 0
\end{array}\right.\end{equation}
A \tgerman{Positivstellensatz} witness is $y^2(-2+y^2)+1(1-y^4)+2y^2=-1$. Another is $\left(\frac{2}{3}+\frac{y^2}{3}\right)(-2+y^2)+\frac{1}{3}(1-y^4)=-1$.

\subsection{Sum-of-Squares Decomposition for the Wide Inequality Case}
Consider the conjunction $C$: $P_1 \geq 0 \wedge \dots \wedge P_n \geq 0$ where $P_i \in \bbQ[X, Y, Z, \dots]$. Consider the set $S$ of products of the form $\prod_{w \in \{0,1\} ^ {\{1,\dots,n\}}} P_i^{w_i}$ --- that is, the set of all products of the $P_i$ where each $P_i$ appears at most once. Obviously, if one can exhibit nonnegative functions $Q_R$ such that $\sum_{R \in S} Q_R R = -1$, then $C$ does not have solutions.
Theorem~\ref{th:real-nullstellensatz} guarantees that if $C$ has no solutions, then such functions $Q_R$ exist as sum of squares of polynomials.
Lemma~\ref{lem:sos_is_sym_matrix} ensures that each $Q_R$ can be expressed as $M_R \hat{Q}_R M_R^T$ where $\hat{Q}_R$ is a symmetric positive semidefinite matrix (noted $\hat{Q}_R \succeq 0$) and $M_R$ is a vector of monomials.

Assume that we know the $M_R$, but we do not know the matrices $\hat{Q}_R$. The equality $\sum_{R \in S} M_R Q_R(M_R)^t R=-1$ directly translates into a system of affine linear equalities over the coefficients of the $\hat{Q}_R$: $\sum_{R \in S} M_R Q_R(M_R)^t R + 1$ is the zero polynomial, so its coefficients, which are linear combinations of the coefficients of the $Q_R$ matrices, should be zero.

The additional requirement is that the $\hat{Q}_R$ are positive semidefinite. One can equivalently express the problem by grouping the $(\hat{Q}_R)_{R \in S}$ matrices into a block diagonal matrix $\hat{Q}$ and express  $\sum_{R \in S} Q_R R$ as a system of affine linear equalities over the coefficients of $\hat{Q}$. By Gaussian elimination in exact precision, we can obtain a system of generators: $\hat{Q} \in -F_0 + \textrm{vect}(F_1, \dots, F_m)$. The only issue is then to find a positive semidefinite matrix in this space; that is, find $\alpha_1,\dots,\alpha_m$ such that $-F_0+\sum_i \alpha_i F_i \succeq 0$.

This is the problem of \emph{semidefinite programming}: finding a positive semidefinite matrix within an affine linear variety of symmetric matrices, optionally optimizing a linear form.
\citet{Powers_Wormann_1998}, \citet[chapter~4]{Parrilo_PhD} and others have advocated such kind of decomposition for finding whether a given polynomial is a sum of squares. \citet{harrison-sos} generalized the approach to finding unsatisfiability witnesses.

For instance, the second unsatisfiability witness we gave for constraint system~\ref{eqn:very_simple_example} is defined, using monomials $\{1,y\}$, $1$ and $\{1,y\}$, by:
\begin{equation*}
\left(\begin{array}{c|c|c}
\begin{array}{cc}
\frac{2}{3}& 0\\
0 & \frac{1}{3}\\
\end{array}&&\\
\hline
& \frac{1}{3} &\\
\hline
&&
\begin{array}{cc}
0 & 0 \\
0 & 0 \\
\end{array}
\end{array}\right)
\end{equation*}

It looks like finding an unsatisfiability witness for~$C$ just amounts to a semidefinite programming problem. There are, however, three problems to solve:
\begin{itemize}
\item $|S| = 2^n$ can be huge.
\item We do not know the degree of the $Q_R$ in advance, so we cannot choose finite sets of monomials $M_R$. The dimension of the space for $Q_R$ grows quadratically in $|M_R|$.
\item Semidefinite programming algorithms are implemented in floating-point. They might therefore provide matrices  $\hat{Q}$ that are not truly positive semidefinite.
\end{itemize}

Conjunction $C$ has no solution if and only if there exists a set of monomials and associated positive semidefinite matrices verifying some linear relations. Positive semidefiniteness is a semialgebraic property of the matrix coefficients, defined by the nonnegativeness of some polynomials in the matrix coefficients (Lemma~\ref{lem:signs_of_charp}). Thus, $C$ has no solution if and only there is a set of monomials such that some set of wide polynomial inequalities has a solution. We have therefore exhibited a form of duality similar to the one described for the linear case in section~\ref{part:linear_ineq_witnesses}.

Let us first consider the first two problems. \citet{Lombardi_zeros_effectifs,Lombardi_HdR} provides a bound to the degrees of the polynomials necessary for the unsatisfiability certificates, but this bound is nonelementary (asymptotically greater than any tower of exponentials), so it is not of a practical value.
This bound, however, is only needed for the \emph{completeness} of the refutation method: we are guaranteed to find the certificate if we look in a large enough space. It is not needed for \emph{soundness}: if we find a correct certificate by looking in a portion of the huge search space, then that certificate is correct regardless. This means that we do not need to consider the whole of~$S$, and we can limit the choice of monomials in~$M_R$ to small degrees without losing soundness.

The third problem is more arduous. Here, problems occur when the
$\hat{Q}$ matrix provided by the semidefinite programming procedure
has eigenvalues that are null or at least very close to zero. Due to
rounding errors, some of these eigenvalues may be slightly negative;
exact computations on such a matrix will find it not to be
positive semidefinite. We shall show in the next section that this
problem is essential and cannot be resolved by augmenting precision:
in many cases, the semidefinite programming problem is degenerate and
solving it involves hitting a hyperplane or some subspace thereof. Since
these objects have infinite thinness, this is impossible numerically
except in some lucky cases.

\section{Degeneracy}
In this section, we shall characterize degeneracy in the semidefinite programming problem. In a nutshell, direct numerical resolution is possible only if the solution set has a nonempty interior: if one finds a solution, then there is a ball of solutions around it, so small roundoff errors may not matter. In contrast, if the solution set has empty interior, then it is included within a hyperplane or some smaller subspace. Except in some rare cases, it is impossible to hit exactly on that plane (for instance, with binary floating point, it is impossible to hit on $x=2/3$). This makes the results from numerical computations unsuitable for being Positivstellensatz witnesses, even if they are close to an exact solution. Furthermore, most numerical methods are interior point methods and fail altogether to provide a numerical solution when the problem is too degenerate.

\subsection{Solution Set of the Semidefinite Programming Problem}
\label{part:solution_set}
Let $F_0$, $F_1,\dots,F_m$ be symmetric $n \times n$ matrices over a subfield $K$ of $R$. The semidefinite programming problem is: find $\lambda_1,\dots,\lambda_m$ such that
\begin{equation}
F = -F_0 + \sum_i \lambda_i F_i \succeq 0
\end{equation}

We may characterize the solution set for $(\alpha_1,\dots,\alpha_m)$ in two ways:
\begin{itemize}
\item For all $v$, $v^t F v \geq 0$, thus $\sum_i (v^t F_i v) \alpha_i \geq v^t F_0 v$, defining a closed half-space. The solution set, being an intersection of closed half-spaces, is therefore convex and closed.
\item The set of positive semidefinite matrices is defined by the sign of the coefficients of the characteristic polynomial (see Lemma~\ref{lem:signs_of_charp}), which are polynomials in $\alpha_i$. Thus, the solution set is semialgebraic.
\end{itemize}

The solution set may have nonempty or empty interior. Its interior corresponds to positive definite solutions, while its boundary corresponds to degenerate positive matrices.

Most semidefinite programming methods are \emph{interior point} methods \citep{Semidefinite_Programming_96}. These methods consider both a primal and a dual problem and assume that both are strictly feasible; the primal being strictly feasible corresponds to a nonempty interior. The problem of finding $\alpha_1,\dots,\alpha_m$ such that $-F_0 + \sum_i \alpha_i F_i \geq 0$ is equivalent to the problem of minimizing $\mu \geq 0$ such that $-F_0 + \sum_i \alpha_i F_i + \mu \textrm{Id} \succeq 0$.

Assume the $-F_0 + \sum_i \alpha_i \succ 0$ strict problem has solutions. The problem then has nonempty interior, and, aside from numerical precision issues, numerical methods should find a solution. The solution set for the strict problem is open; if there is a real solution, then within a small ball around it all rational points are also solutions. Assuming enough precision, the problem is then solved.

In general, though, the solution set may have empty interior. Equivalently, the least enclosing linear affine variety (the Zariski closure of the solution set) may not have full dimension. As an example, consider:
\begin{equation}
\label{eq:ex_line}
\begin{array}{cc}
-F_0 =
\begin{pmatrix}
 -\frac{130555}{143} & -\frac{150364}{91} & -\frac{19213}{7} \\
 -\frac{150364}{91} & -\frac{1883353}{1001} & -\frac{41326}{13} \\
 -\frac{19213}{7} & -\frac{41326}{13} & -\frac{767287}{143}
\end{pmatrix}
&
\begin{pmatrix}
F_1 =
 105 & 89 & 153 \\
 89 & 95 & 161 \\
 153 & 161 & 273
\end{pmatrix}
\\
F_2 =
\begin{pmatrix}
 129 & 110 & 187 \\
 110 & 49 & 88 \\
 187 & 88 & 157
\end{pmatrix}
&
F_3 =
\begin{pmatrix}
 49 & 86 & 143 \\
 86 & 97 & 164 \\
 143 & 164 & 277
\end{pmatrix}
\\
\end{array}
\end{equation}

The solution set is a segment (of positive length) of the line defined by $\alpha_2=-3/11$ and $91 (\alpha_1 + \alpha_3)=1811$. If we recast the problem on this line, the solution set has nonempty interior. Unfortunately, we know of no easy way to obtain the equations of this enclosing linear variety in the general case. We can however provide some partial solutions to this problem.

The solution set $S$ has empty interior while being non empty if and only if the linear affine variety $-F_0 + \textrm{vect}(F_1, \dots, F_m)$ is tangent to the $\det F=0$ variety. This means that the differential of $\phi: (\alpha_1, \dots, \alpha_m) \mapsto \det(F_0 + \sum_i \alpha_i F_i)$ is null at the solution point, that is, we are at a singular point of the variety defined by this polynomial.

In the case of example~\ref{eq:ex_line}, $\phi=0$ and $\partial \phi / \partial \alpha_i = 0$ yield four equations. By Gr\"obner basis techniques followed by polynomial factorization we can obtain $(3 + 11 l2)^2=0$ and $91 (\alpha_1 + \alpha_3)=1811$.

Yet, in the general case, things are not so simple. Consider the following example:\footnote{Courtesy of Kevin Buzzard.}
\begin{equation}
\begin{array}{ccc}
F_0 = 0 &
F_1 = \begin{pmatrix}
0 & 1 & 0 & 0\\
1 & 0 & 0 & 0\\
0 & 0 & 0 & 1\\
0 & 0 & 1 & 0\\
\end{pmatrix} &
F_2 = \begin{pmatrix}
1 & 0 & 0 & 0\\
0 & -1 & 0 & 0\\
0 & 0 & 1 & 0\\
0 & 0 & 0 & -1
\end{pmatrix}
\end{array}
\end{equation}

$\det F = (\alpha_1^2+\alpha_2^2)^2$. Gr\"obner basis and factorization techniques will yield $\alpha_1^2+\alpha_2^2=0$. Even if we replaced $F_1$ and $F_2$ by another basis, we would still obtain a second degree homogeneous polynomial, which can be transformed into a sum of squares (Lemma~\ref{lem:sym-matrix-is-sos}). Now consider:
\begin{equation}F'_0 = \begin{pmatrix}
9 & -5 & 0 & 0 \\
 -5 & -7 & 0 & 0 \\
 0 & 0 & 7 & -5 \\
 0 & 0 & -5 & -7
\end{pmatrix}
\end{equation}

$F = -F'_0 + \alpha_1 F_1 + \alpha_2 F_2 \succeq 0$ has a unique solution ($\alpha_1=5$, $\alpha_2=-7$), where
\begin{equation}F = \begin{pmatrix}
2 & 0 & 0 & 0 \\
 0 & 0 & 0 & 0 \\
 0 & 0 & 0 & 0 \\
 0 & 0 & 0 & 0
\end{pmatrix}
\end{equation}

This may be found algebraically, by constraining the signs of the coefficients of the characteristic polynomial of~$F$. Yet, in this very degenerate case where the solution is a single point within a plane, with a corresponding rank-1 matrix, neither DSDP5 \citep{DSDP_manual} nor SDPA \citep{SDPA_manual}, two semidefinite programming packages, can compute an approximation to the solution.

\subsection{Degenerate Positivstellensatz Problem}
To make constraint system~\ref{eqn:very_simple_example} more interesting, we replace $y$ by $3a + b + 1$, which yields
\begin{equation}\label{eqn:more_complex_example}
\left\{\begin{aligned}
0 \leq P_1 = & 9 a^2+6 b a+6 a+b^2+2 b-1\\
0 \leq P_2 = & -81 a^4-108 b a^3-108 a^3-54 b^2
   a^2-108 b a^2-54 a^2-12 b^3 a \\
  & -36 b^2 a-36 b a-12 a-b^4-4 b^3-6
   b^2-4 b\geq 0
\end{aligned}\right.
\end{equation} 

We look for a witness of the form $Q_1(1,a,b) P_1 + Q_2(1) P_2 + Q_3(1,a,b,ab) P_3 = -1$. We group $Q_1$, $Q_2$ and $Q_3$ into a single block diagonal matrix:

\begin{equation*}
\left(\begin{array}{c|c|c}
\begin{array}{ccc}
&\vdots&\\
\hdots&Q_1&\hdots\\
&\vdots&\\
\end{array}&&\\
\hline
& Q_2 &\\
\hline
&&
\begin{array}{cccc}
&\vdots&&\\
\hdots&Q_3&\hdots&\hdots\\
&\vdots&&\\
\end{array}
\end{array}\right)
\end{equation*}

$Q$ belongs to a linear affine variety defined as $-F_0+\textrm{vect}(F_1,F_2,F_3)$.
\begin{equation*}
-F_0 = \left(\begin{array}{c|c|c}
\begin{matrix}
 -2 & -\frac{11}{2} & -\frac{11}{6}
   \\
 -\frac{11}{2} & -\frac{33}{2} &
   -\frac{11}{2} \\
 -\frac{11}{6} & -\frac{11}{2} &
   -\frac{11}{6}
\end{matrix}&&
\\\hline
& \begin{matrix}
 -\frac{11}{6}
\end{matrix}&
\\\hline
&&\begin{matrix}
 -3 & -\frac{21}{2} & -\frac{7}{2} &
   -10 \\
 -\frac{21}{2} & -\frac{63}{2} &
   -\frac{1}{2} & 0 \\
 -\frac{7}{2} & -\frac{1}{2} &
   -\frac{7}{2} & 0 \\
 -10 & 0 & 0 & 0
\end{matrix}
\end{array}\right)
\end{equation*}

\begin{equation*}
F_1 = \left(\begin{array}{c|c|c}
\begin{matrix}
 0 & 3 & 1 \\
 3 & 9 & 3 \\
 1 & 3 & 1
\end{matrix}&&\\
\hline
&\begin{matrix}
 1
\end{matrix}&\\
&&\begin{matrix}
 0 & 9 & 3 & 9 \\
 9 & 27 & 0 & 0 \\
 3 & 0 & 3 & 0 \\
 9 & 0 & 0 & 0
\end{matrix}
\end{array}\right)
\end{equation*}

\begin{equation*}
F_2 = \left(\begin{array}{c|c|c}
\begin{matrix}
 0 & 0 & 0 \\
 0 & 0 & 0 \\
 0 & 0 & 0
\end{matrix}&&\\
\hline
&\begin{matrix}
 0
\end{matrix}&\\
\hline
&&\begin{matrix}
 0 & 0 & 0 & -1 \\
 0 & 0 & 1 & 0 \\
 0 & 1 & 0 & 0 \\
 -1 & 0 & 0 & 0
\end{matrix}
\end{array}\right)
\end{equation*}

\begin{equation*}
F_3 = \left(\begin{array}{c|c|c}
\begin{matrix}
 3 & 3 & 1 \\
 3 & 9 & 3 \\
 1 & 3 & 1
\end{matrix}&&\\
\hline
&\begin{matrix}
 1
\end{matrix}&\\
\hline
&&\begin{matrix}
 3 & 0 & 0 & 0 \\
 0 & 0 & 0 & 0 \\
 0 & 0 & 0 & 0 \\
 0 & 0 & 0 & 0
\end{matrix}
\end{array}\right)
\end{equation*}

All $(\alpha_1, \alpha_2, \alpha_3)$ solutions ($Q \succeq 0$) verify $-9\alpha_1+\alpha_2=-10$ (this was obtained through algebraic methods). As explained in section~\ref{part:solution_set}, this plane is the Zariski closure of the solution set. An example of a solution is $\alpha_1=2, \alpha_2=8, \alpha_3=79$. Unfortunately, neither SDPA nor DSDP can compute such a result. Both terminate``due to small steps''.

We have therefore exhibited a simple system with two parameters where, due to the emptiness of the solution set, numerical interior point methods fail, while algebraic methods can compute a solution point.

Assuming we have a method for obtaining the Zariski closure ($-9\alpha_1+\alpha_2=-10$), then we can use it to reduce the system. By rewriting $\alpha_2=-10+9\alpha_1$, we obtain a system $F'_0,F'_1,F'_2$, with a solution set with nonempty interior, and numerical solving works.

Algorithms for computing the Zariski closure of a semialgebraic set should be at least as complex as those for finding a single solution point, if only because, in the case of a solution set consisting of a single point, the Zariski closure is equal to that point. Yet, the Zariski closure is only useful so as to help numerical methods find solution points, so computing this closure by computing solution points or equally complex computations defeats the purpose.

With more complex examples (more polynomials, larger monomial bases), the number of $\alpha_i$ coefficients grows dramatically (in the hundreds). Computing the determinant of a symbolic matrix $Q_i$ may become untractable. Algebraic methods for computing solution points are then infeasible, since they rely on the sign of the determinant.

Assuming the numerical method does not fail and produces a good approximation $\tilde{\alpha}_1,\dots,\tilde{\alpha}_m$ to a rational solution, one can use several methods to help compute the rational solution. The most obvious one is to find rational approximations to the floating-point by e.g. continued fractions; yet this fails to obtain a solution in most cases. If the Zariski closure has dimension $z < m$, assuming this closure is not parallel to the $\alpha_i=K$ plane and the approximation is good enough, then by choosing $\alpha_i=\tilde{\alpha_i}$ one ``slices'' the problem down to finding a point within a $z-1$-dimensional solution set within a $m-1$-dimension space. If one does that with many variables, one obtains a $0$-dimension solution set (a single point) within a $z-m$ space. Then, the problem has empty interior, and cannot be solved numerically in general: only algebraic methods are feasible. 

\section{Conclusion}
The approach of finding unsatisfiability witnesses for real polynomial inequalities through \tgerman{Positivstellensatz} and reduction to semidefinite programming looked promising. Unfortunately, it suffers from several drawbacks:
\begin{enumerate}
\item If one has $n$ polynomial inequalities, then one has to consider at most $2^n$ terms in the sum expressing the unsatisfiability witness.
\item There is no reasonable known bound on the size of the monomial bases to consider.
\item In general, one gets a degenerate semidefinite programming problem --- that is, a problem whose solution set has no interior point. Numerical interior point methods in general fail to converge on such problems. Even if they do provide an approximate solution, this solution cannot be easily mapped to an exact rational solution. It is possible to get rid of this problem by going into lower dimensions, however this involves computing the Zariski closure of the solution set, which may be as difficult as finding a solution point. This defeats the purpose of using numerical methods, which was to avoid costly algebraic algorithms.
\end{enumerate}

\section*{Acknowledgements}
We wish to thank Alexis Bernadet and Mohab Safey El Din for their helpful comments and ideas.

\bibliographystyle{plainnat}
\bibliography{decision_of_polynomial_inequalities}
 
\appendix
\section{Lemmas}
\subsection{Sums of Squares and Symmetric Matrices}
\begin{lemma}\label{lem:vtv-psd}
Let $v \in K^n$. Then, $v^T v$ is a $n \times n$ symmetric positive semidefinite matrix.
\end{lemma}

\begin{proof}
$v^T v$ is obviously symmetric. Let $\lambda$ be a eigenvalue for it, and $x$ a corresponding eigenvector: $x v^T v = \lambda x$. Thus, $\|v x^T\|_2^2 = (v x^T)^T(v x^T) = x v^T v x^T = \lambda x x^T = \| x \|^2_2$. Since $x \neq 0$, $\lambda$ must be nonnegative.
\end{proof}

\begin{lemma}\label{lem:sos_is_sym_matrix}
Let $P \in K[X, Y, \dots]$ be a sum of squares of polynomials $\sum_i P_i^2$. Let $M = \{m_1, \dots, m_{|M|}\}$ be a set such that each $P_i$ can be written as a linear combination of elements of~$M$ ($M$ can be for instance the set of monomials in the $P_i$). Then there exists a $|M| \times |M|$  symmetric positive semidefinite matrix $Q$ with coefficients in $K$ such that $P(X, Y, \dots) = [m_1, \dots, m_{|M|}] Q [m_1, \dots, m_{|M|}]^T$, noting $v^T$ the transpose of~$v$.
\end{lemma}

\begin{proof}
Let us decompose $P_i (X, Y, \dots)$ into a linear combination of monomials $\sum_{1 \leq j \leq |M|} p_{i,j} m_j$. Let $v_i$ be the vector $[p_{i,1}, \dots, p_{i, m}]$; then $P_i (X, Y, \dots) = v_i [m_1, \dots, m_{|M|}]^T$.
$P_i^2 (X, Y, \dots)$ is thus $[m_1, \dots, m_{|M|}] v_i^T v_i [m_1, \dots, m_{|M|}]^T$. $Q_i = v_i^T v_i$, by lemma~\ref{lem:vtv-psd} is symmetric positive semidefinite. $Q = \sum_i Q_i$ thus fulfills the conditions.
\end{proof}

Let us remark that the converse is correct for matrices over~$\bbR$, by diagonalization: any symmetric positive semidefinite matrix is a sum of squares of linear forms. We may also obtain such a decomposition over~$\bbQ$:

\begin{lemma}\label{lem:sym-matrix-is-sos}
Let $Q$ be a $n \times n$ symmetric matrix over a subfield $K$ of $\bbR$. Then $Q$ can be written as $U^t D U$ where $U$ and $D$ are also over $K$, $D$ is diagonal and $U$ is upper triangular. Otherwise said, $(x_1, \dots, x_n) \mapsto (x_1, \dots, x_n)^t Q (x_1, \dots, x_n)$ can be written as $\sum_{i=1}^n d_i l_i(x_1,\dots,x_n)^2$ where $l_i$ is a linear form and only depends on $x_1,\dots,x_i$.
Furthermore, $D$ and $U$ have the same signature; in particular if $Q$ is positive semidefinite, then $D$ only has nonnegative coefficients
\end{lemma}

\begin{proof}
By induction over $n$. The case $n=1$ is obvious; consider $n > 1$. Let $l=\left(1,\frac{q_{1,2}}{q_{1,1}},\dots,\frac{q_{1,n}}{q_{1,1}}\right)$. $Q_1=Q - q_{1,1} l^t l$ contains only zeroes on its first line and column. By the induction hypothesis, $Q_1 = U_1^t D_1 U_1$. Let $D = (q_{1,1}, D_1)$ (concatenation along the diagonal) and $U = (l,U_1)$ (concatenation of lines), then $Q=U^t D U$. The result on signatures ensues from Sylvester's inertia theorem.
\end{proof}

\subsection{Semialgebraic Characterization of Positive Semidefinite Matrices}
\begin{lemma}\label{lem:nonnegative_roots}
Let $\sigma_i(X_1, \dots, X_n)$, where $1 \leq i \leq n$, denote the $i$-th elementary symmetric polynomial in the variables $X_1, \dots, X_n$. $x_1, \dots, x_n$ are all nonnegative if and only if $\sigma_1(x_1, \dots, x_n), \dots, \sigma_n(x_1, \dots, x_n)$ are so
\end{lemma}

\begin{proof}
One direction is evident: if $x_1, \dots, x_n$ are nonnegative, then  $\sigma_1(x_1, \dots, x_n), \dots, \sigma_n(x_1, \dots, x_n)$ also are nonnegative, for these polynomials have nonnegative coefficients.

Let us suppose that $\sigma_1(x_1, \dots, x_n), \dots, \sigma_n(x_1, \dots, x_n)$ are nonnegative, and that at least one of them is positive.
$x_1,\dots,x_n$ are the roots of the polynomial $P(X)=X^n+\sum_{i=1}^n (-1)^i \sigma_i(x_i,\dots,x_n) X^{n-i}$. For $\rho < 0$, $P(\rho)>0$ by the rule of signs, so this polynomial has no negative roots thus $x_1,\dots,x_n$ are nonnegative.

The last case is where $\sigma_1(x_1, \dots, x_n)=\dots=\sigma_n(x_1, \dots, x_n)=0$. Since $\sigma_n(x_1,\dots,x_n)=x_1 \dots x_n=0$, this means at least one of the $x_i$ is null. The problem reduces to the same with a lower~$n$.
\end{proof}

\begin{lemma}\label{lem:sign_symmetric_polys}
Let $x_1,\dots,x_n$ be nonnegative reals. Then, the sequence $(\sigma_i(x_1,\dots,x_n))_{0 \leq i < n}$ consists in $k$ zeroes followed by $n-k$ positive reals where $k$ is the number of zeroes among $x_1,\dots,x_n$.
\end{lemma}

\begin{proof}
Obvious.
\end{proof}

\begin{lemma}\label{lem:signs_of_charp}
Let $M$ be a $n \times n$ real symmetric matrix. Let $\chi_M(X) = \det(M-X.\textrm{Id}) = \sum_{i=0}^n p_i X^n$ be its characteristic polynomial. Then $M$ is positive semidefinite if and only if for all $0 \leq i < n$, $(-1)^i p_i \geq 0$. Furthermore, the sequence $(-1)^i p_i$ consists in $\dim \ker M $ zeroes followed by $n-k$ positive numbers. 
\end{lemma}

\begin{proof}
For all $0 \leq i < n$, $p_i = (-1)^i \sigma_{n-i}(\lambda_1,\dots,\lambda_n)$ where the $(\lambda_i)_{1 \leq i \leq n}$ are the eigenvalues of~$M$ (multiple eigenvalues are counted as several $\lambda_i$). The result then ensues from lemmas~\ref{lem:nonnegative_roots} and \ref{lem:sign_symmetric_polys}.
\end{proof}
\end{document}